\documentclass{article}
\usepackage[utf8]{inputenc}
\usepackage{amsmath,amssymb,amsthm,bbm}
\usepackage{color}
\usepackage{mathtools}
\DeclareMathOperator\prox{P}
\DeclareMathOperator*{\argmin}{arg\,min}

\DeclareMathOperator*{\R}{\mathbb{R}}
\DeclareMathOperator*{\Rn}{\mathbb{R}^n}
\DeclareMathOperator*{\Rm}{\mathbb{R}^m}

\newtheorem{theorem}{Theorem}

\newtheorem{lemma}[theorem]{Lemma}
\newtheorem{assumption}{Assumption}

\title{An Envelope for Davis-Yin Splitting and Strict Saddle Point Avoidance}
\author{Yanli Liu$^*$ \and Wotao Yin\footnote{Department of Mathematics, University of California, Los Angeles, CA 90095.\newline E-mails: yanli / wotaoyin@math.ucla.edu}}
\date{\today}

\begin{document}

\maketitle

\begin{abstract}
    It is known that operator splitting methods based on Forward Backward Splitting (FBS), Douglas-Rachford Splitting (DRS), and Davis-Yin Splitting (DYS) decompose a difficult optimization problems into simpler subproblem under proper convexity and smoothness assumptions. In this paper, we identify an envelope (an objective function) whose gradient descent iteration under a variable metric coincides with DYS iteration. This result generalizes the Moreau envelope for proximal-point iteration and the envelopes for FBS and DRS iterations identified by Patrinos, Stella, and Themelis. 
    
    Based on the new envelope and the Stable-Center Manifold Theorem, we further show that, when FBS or DRS iterations start from random points, they avoid all strict saddle points with probability one. This result extends the similar results by Lee et al. from gradient descent to splitting methods.
\end{abstract}

\section{Introduction}

The most general model considered in this paper minimizes the sum of three functions, where two of them are Lipschitz differentiable and, out of these two, one can involve a composition with a linear operator. The third function can be non-differentiable, and all the three functions can be nonconvex. A mathematical formulation is given in Section \ref{DYS envelope}. The results of this paper, of course, apply to simpler models where any one or two of these three functions vanish. Problems that can be written in our general model are abundant. Examples include texture inpainting \cite{liu2013tensor}, matrix completion \cite{candes2010matrix}, and support vector machine classification \cite{cortes1995support}.

Our model can be solved by the splitting iterative methods based on Douglas-Rachford Splitting (DRS) \cite{lions1979splitting} and Forward-Backward Splitting (FBS) \cite{passty1979ergodic}, as well as their generalization, Davis-Yin Splitting (DYS) \cite{davis2017three}. In these methods, the problem objective is split into different steps, one of each of the objective functions. Their implementations are typically straightforward. By exploiting additional sum and coordinate friendly structures, they give rise to parallel and distributed algorithms that are highly scalable. The details of these methods are reviewed in Section \ref{DYS envelope} below.

These splitting methods are traditionally analyzed assuming that the subdifferentials of the objective functions are maximally monotone. Those of nonconvex functions are generally non-monotone. Therefore, the majority of the existing results apply to only convex objective functions.

Recently, FBS and DRS are found to numerically converge for certain nonconvex problems, for example, FBS for image restoration \cite{stella2017forward}, dictionary learning, and matrix decomposition \cite{themelis2016forward}, and 
DRS for nonconvex feasibility problem \cite{li2016douglas}, matrix completion \cite{artacho2014douglas}, and phase retrieval \cite{chen2016fourier}. Theoretically, their iterates have been shown to converge to stationary points in some nonconvex settings \cite{attouch2013convergence, li2016douglas, themelis2017douglas, guo2017convergence}. In particular, any bounded sequence produced by FBS converges to a stationary point when the objective satisfies the KL property \cite{attouch2013convergence}; By using the Douglas-Rachford Envelope (DRE), DRS iterates are shown to converge to a stationary point when one of the two functions is Lipschitz differentiable, both of them are semi-algebraic and bounded below, and one of them is coercive \cite{li2016douglas,themelis2017douglas}; In \cite{guo2017convergence}, when one function is strongly convex and the other is weakly convex, and their sum is strongly convex, DRS iterates are shown to be Fejer monotone with respect to the set of fixed points of DRS operator, thus convergent. Though unlikely, it is still possible that the limit of a convergent sequence is a saddle point instead of a local minimum (except when all stationary points are local minima, which is the case studied in \cite{guo2017convergence}). This depends on the problem geometry and the selected start point.

Recently, some first-order methods have been shown to avoid so-called strict saddle points, with probability one regarding random initialization \cite{lee2016gradient,lee2017first}. These results make skillful use of the Stable-Center Manifold Theorem \cite{shub2013global}. So far, their results apply to only relatively simple methods such as Gradient Descent, Coordinate Descent, and Proximal Point methods. 
We give an affirmative answer (under smoothness assumptions) that splitting methods also have this property. This result also matches the practical observations made in~\cite{sun2017complete}.

This paper makes the following contribution regarding the convergence and saddle point avoidance of FBS, DRS, and DYS iterations for nonconvex problems. We first generalize the existing Forward-Backward Envelope (FBE) and Douglas-Rachford Envelope (DRE) into a Davis-Yin Envelope (DYE) and establish relationships between the latter envelope and the original optimization objective. Then,  under smoothness conditions, we show that the probability for DRS and FBS iterations with random initializations converge to strict saddle points of their respective DRE and FBE is zero. Finally, by the connection between the envelopes and the original objectives, we extend the above avoidance results to the strict saddle points of the original objectives. That is, when our problem has the strict saddle property,  DRS and FBS with random initialization will almost surely converge to local minimizers. The strict saddle property is satisfied in several applications including, but not limited to, dictionary learning~\cite{sun2017complete}, simple neural networks~\cite{brutzkus2017globally},  phase retrieval~\cite{sun2016geometric}, tensor decomposition~\cite{ge2015escaping}, and low-rank matrix factorization~\cite{bhojanapalli2016global}.

The rest of this paper is organized as follows. In Section~\ref{pre}, we introduce notation and review some useful results. In Section~\ref{DYS envelope}, we define the envelope for DYS and rewrite DYS equivalently as gradient descent of this envelope. In Section~\ref{properties of the envelope}, we establish a strong relationship between the envelope and the objective. Then, in Section~\ref{avoidance}, we analyze the avoidance of strict saddle points of the objective. Finally, we conclude this paper in Section~\ref{conclusion}.

\section{Preliminaries}
\label{pre}
In this section, we review some basic concepts, introduce our notation, and state some known results. For the sake of brevity, we omit proofs and direct references. We refer the reader to textbooks~\cite{rockafellar2009variational,bauschke2011convex}.

We let $\mathbf{0}\in\Rn$ denote the vector zero,  $\langle\cdot,\cdot\rangle$ the inner product, $\|\cdot\|$ the $\ell_2$ norm, and Fix$T$  the set of fixed points of operator $T$.

A function $f:\Rn \rightarrow \mathbb{R}\cup\{\infty\}$ is called $\beta-$weakly convex (or $\beta-$semiconvex) if the function $\tilde{f}:=f+\frac{\beta}{2}\|\cdot\|^2$ is convex. Clearly, $f$ can be nonconvex.

Let $y\xrightarrow{f} x$ denote $y \rightarrow x$ and $f(y) \rightarrow f(x)$. Then the subdifferential of $f$ at $x\in \mathrm{dom}\,f$ can be defined by
\begin{align*}
\partial f(x)\coloneqq&\Big\{v\in \Rn\,\big|\, \exists x^t\xrightarrow{f} x, v^t\rightarrow v, \mathrm{with}\\ &\quad\liminf_{z\rightarrow x^t}\frac{f(z)-f(x^t)-\langle v^t, z-x^t\rangle}{\|z-x^t\|}\geq 0 \,\,\mathrm{for \,\,each}\,\,t\Big\}.
\end{align*}
If $f$ is differentiable at $x$, we have $\partial f(x)=\{\nabla f(x)\}$;
If $f$ is convex, we have
\[
\partial f(x)=\{v\in \Rn\,|\, f(z)\geq f(x)+\langle v, z-x\rangle\,\,\mathrm{for\,\, any}\,\,z\in\Rn\},
\]
which is the classic definition of subdifferential in convex analysis.

A point $x^*$ is a stationary point of a function $f$ if $\mathbf{0}\in\partial f(x^*)$. $x^*$ is a critical point of $f$ if $f$ is differentiable at $x^*$ and $\nabla f(x^*)=\mathbf{0}$.

A point $x^*$ is a \emph{strict saddle point} of $f$ if $f$ is twice differentiable at $x^*$, $x^*$ is a critical point of $f$, and $\lambda_{\min}[\nabla^2 f(x^*)]<0$, where $\lambda_{\min}[\cdot]$ returns the smallest eigenvalue of the input. Local minimizers of a function are always its stationary points, but not strict saddle points. 

For any $\gamma>0$, the proximal mapping of a function $f$ is defined by
\[
\prox_{\gamma f}(x):=\argmin_{y\in\Rn}\{ f(y)+\frac{1}{2\gamma}\|y-x\|^2\},
\]
assuming that the right-hand side exists. 
When $f$ is convex, $\prox_{\gamma f}$ is single-valued and satisfies
$\prox_{\gamma f}(x)=(\mathrm{Id}+\gamma \partial f)^{-1}(x),$
where Id is the identity map.
For any proper, closed, convex function $f$, it Moreau Identity is
\begin{equation}
    \label{MoreauIdentity}
    \mathrm{Id}=\prox_{\gamma f}+\gamma\prox_{\frac{f^*}{\gamma}}\circ\frac{\mathrm{Id}}{\gamma},
\end{equation}
where 
$
f^*(u)\coloneqq\sup_{x\in\Rn}\{\langle u,x\rangle-f(x)\}
$
is the convex conjugate of $f$.

We also need the Inverse Function Theorem: for a $C^1$ mapping $F: \Rn\rightarrow \Rn$, if the Jacobian $J_{F}(x)$ of $F$ at $x\in\Rn$ is invertible, then there exists an inverse function $F^{-1}$ defined in a neighbourhood of $F(x)$ such that $F^{-1}$ is also $C^{1}$ and 
\begin{equation}
J_{F^{-1}}\big(F(x)\big)=\big(J_{F} (x)\big)^{-1}.
\label{inversefunctionthm}
\end{equation}

\section{Envelope for Davis-Yin Splitting}
\label{DYS envelope}

In this section, we will introduce a function, which we call an envelope, such that DYS iteration can be written as the gradient descent of this function under a variable metric.
Since DYS generalizes FBS and DRS, the envelope of DYS is also a generalization of FBE and DRE, the respective envelopes of FBS and DRS, which were introduced in \cite{stella2017proximal,themelis2016forward}. 

\subsection{Review of Davis-Yin splitting}\label{dys_review}
DYS~\cite{davis2017three} can be applied to solve the following problem:
\begin{equation}
\label{objective}
\mathop{\mathrm{minimize}}_{x\in \mathbb{R}^n} \varphi(x) \coloneqq f(x)+g(x)+h(Lx),
\end{equation}
where, for this subsection, $f,g: \Rn\rightarrow\R$ and $h: \Rm\rightarrow\R$ are proper, closed, and convex,
$h$ is also $L_h-$Lipschitz differentiable, and $L:\mathbb{R}^n\rightarrow \mathbb{R}^m$ is a linear operator. 

DYS iteration produces a sequence $(x^k)_{k\ge0}$ according to $z^{k+1} = Tz^k$, where
$$Tz^k:= z^k + \alpha\Big(\prox_{\gamma f}\big(2\prox_{\gamma g}(z^k)-z^k-\gamma L^T\nabla h(Lz^k)\big)-\prox_{\gamma g}(z^k)\Big),$$
where $\gamma$ and $\alpha$ are positive scalars.
We rewrite this operator into successive steps with designated letters as
\begin{align}
q^k&\coloneqq L^T\nabla h(Lz^k),\nonumber\\
r^k&\coloneqq 2\prox_{\gamma g}(z^k)-z^k,\nonumber\\
p^k&\coloneqq \prox_{\gamma f}(r^k-\gamma q^k),\nonumber\\
w^k&=p^k-\prox_{\gamma g}(z^k),\label{w^k}\\
z^{k+1}&=Tz^k=z^k+\alpha w^k.\label{iter}
\end{align}
Since $f,g$ are closed, proper, and convex (in this subsection), $\prox_{\gamma g}$ and $\prox_{\gamma f}$ are well defined and single valued. In \cite{davis2017three}, convergence is established for a range of parameters
\begin{align*}
  \gamma \in \Big(0,\frac{2L_h}{\|L\|^2}\Big)\quad \mathrm{and}\quad \alpha \in \Big(0,2-\frac{\gamma \|L\|^2}{2L_h}\Big).  
\end{align*}
When $h=0$, \eqref{iter} simplifies to Douglas-Rachford Splitting iteration,
\[
z^{k+1}=z^k+\alpha \big(\prox_{\gamma f}(r^k)-\prox_{\gamma g}(z^k)\big).
\]

When $g=0$, $\prox_{\gamma g}$ reduces to Id and thus \eqref{iter} simplifies to
\[
z^{k+1}=z^k+\alpha(\prox_{\gamma f}(z^k-\gamma q^k)-z^k),
\]
which is Forward-Backward Splitting iteration slightly generalized by including the linear operator $L$.

When $f=0$, $\prox_{\gamma f}$ reduces to Id and \eqref{iter} simplifies to Backward-Forward Splitting,
\[
z^{k+1}=z^k+\alpha (\prox_{\gamma g}(z^k)-\gamma q^k-z^k).
\]

When $f=g=0$, \eqref{iter} simplifies to gradient descent iteration
\[
z^{k+1}=z^k-\alpha\gamma q^k.
\]

\subsection{Derivation of envelope}
Now we show that, \eqref{iter} can be written as gradient descent iteration of an envelope function under the following assumption.  
\begin{assumption}
\label{assumption 1}
\hfill
\begin{enumerate}
    \item $g: \Rn\rightarrow {\R}$ is $L_g-$Lipschitz differentiable.
    \item $h: \Rm\rightarrow {\R}$ is $L_h-$Lipschitz differentiable.
    \item $f: \Rn\rightarrow {\R}\cup\{\infty\}$ is proper, lower semicontinuous and $\gamma\in(0,\frac{1}{L_g+L_h\|L\|^2})$. In addition, $f$ is prox-bounded in the sense that $f(\cdot)+\frac{1}{2\gamma}\|\cdot\|^2$ is lower bounded for any  $\gamma\in(0,\frac{1}{L_g+L_h\|L\|^2})$.
\end{enumerate}
\end{assumption}
Compared to the assumption in Section~\ref{dys_review}, a main restriction is that $g$ is Lipschitz differentiable. On the other hand, all $f$, $g$ and $h$ can be nonconvex.

We begin with two technical lemmas regarding the Moreau envelope of weakly convex functions and its twice differentiability.
\begin{lemma} 
\label{Moreau}
Let $\xi$ be proper, closed, $\beta-$weakly convex. Choose $\gamma$ such that $\gamma\in(0, \frac{1}{\beta})$. Let $\xi^{\gamma}(z)=\min_{u\in \mathbb{R}^n}\{\xi(u)+\frac{1}{2\gamma}\|z-u\|^2\}$ be the Moreau envelope of $\xi$. Define $\tilde{\xi}=\xi+\frac{\beta}{2}\|\cdot\|^2$, which is convex. Then, proximal mapping $\prox_{\gamma \xi}(z)$ is single-valued and satisfies
\begin{align*}
\prox_{\gamma \xi}(z)&=\prox_{\frac{\gamma}{1-\gamma \beta}\tilde{\xi}}(\frac{1}{1-\gamma \beta}z),\\
\nabla \xi^{\gamma}(z)&=\gamma^{-1}\big(z-\prox_{\gamma \xi}(z)\big).
\end{align*}
Furthermore, $\prox_{\gamma \xi}(z)$ is $\frac{1}{1-\gamma \beta}-$Lipschitz continuous.
\end{lemma}

\begin{proof}
We have
\begin{align*}
    \xi^{\gamma}(z)
    &=\min_{u\in\R} \{\xi(u)+\frac{\beta}{2}\|u\|^2+\frac{1}{2\gamma}\|u-z\|^2-\frac{\beta}{2}\|u\|^2\}\\
    &=\min_{u\in\R}\Big\{\tilde{\xi}(u)+\frac{1-\gamma \beta}{2\gamma}\|u-\frac{1}{1-\gamma \beta}z\|^2\Big\}-\frac{\beta}{2-2\gamma \beta}\|z\|^2
\end{align*}
where the second equality follows from the definition of $\tilde{\xi}$.

As a result, for $\gamma\in(0,\frac{1}{\beta})$, $\prox_{\gamma \xi}$ is single-valued and 
\begin{align*}
    \prox_{\gamma \xi}(z)=\prox_{\frac{\gamma}{1-\gamma \beta}\tilde{\xi}}(\frac{1}{1-\gamma \beta}z).
\end{align*}
Since $\prox_{\frac{\gamma}{1-\gamma \beta}\tilde{\xi}}(z)$ is $1-$Lipschitz continuous, we know that $\prox_{\gamma \xi}(z)$ is $\frac{1}{1-\gamma \beta}-$Lipschitz continuous.

Finally, since $\tilde{\xi}$ is convex,  \cite[Prop.12.29]{bauschke2011convex} tells us that $\xi^{\gamma}$ is differentiable and
\begin{align*}
\nabla \xi^{\gamma}(z)
&=\frac{1}{1-\gamma \beta}\nabla \tilde{\xi}^{\frac{\gamma}{1-\gamma \beta}}(\frac{1}{1-\gamma \beta}z)-\frac{\beta}{1-\gamma \beta}z\\
&=\frac{1}{1-\gamma \beta}\frac{1-\gamma \beta }{\gamma}\big(\frac{1}{1-\gamma \beta}z-\prox_{\frac{\gamma}{1-\gamma \beta}\tilde{\xi}}(\frac{1}{1-\gamma \beta}z)\big)-\frac{\beta}{1-\gamma \beta}z\\
&=\frac{1}{\gamma}\big(z-\prox_{\gamma \xi}(z)\big).
\end{align*}
\end{proof}

\begin{lemma}
\label{g related differentiablility}
In addition to Assumption \ref{assumption 1}, if $g$ is twice differentiable at $\prox_{\gamma g}(z^0)$, then 
$\prox_{\gamma g}$ has a Jacobian at $z^0$, $g^{\gamma}$ is twice differentiable at $z^0$ with the Hessian
\[
\nabla^2 g^{\gamma}(z^0)=\frac{1}{\gamma}\bigg(I-\Big(I+\gamma \nabla^2g\big(\prox_{\gamma g}(z^0)\big)\Big)^{-1}\bigg).
\]
In addition, the mapping 
\begin{equation}
A(z)\coloneqq I-2\gamma \nabla^2g^{\gamma}(z)-\gamma L^T\nabla^2h\big((L\prox_{\gamma g}(z)\big)L\big(I-\gamma \nabla^2 g^{\gamma}(z)\big)
\label{A}
\end{equation} 
is invertible.
\end{lemma}

\begin{proof}
Since $\gamma\in(0,\frac{1}{L_g})$, $\prox_{\gamma g}$ is single-valued and 
\begin{equation}
\label{X}
\prox_{\gamma g}(z^0)=(\mathrm{Id}+\gamma \nabla g)^{-1}z^0,
\end{equation}
where $(\mathrm{Id}+\gamma\nabla g)^{-1}$ is the inverse mapping of $\mathrm{Id}+\gamma \nabla g$. Since $\nabla^2 g(\prox_{\gamma g}\big(z^0)\big)$ is symmetric and its eigenvalues are bounded by $L_g$, we know that $I+\gamma \nabla^2 g\big(\prox_{\gamma g}(z^0)\big)$ is invertible, which is the  Jacobian of $\mathrm{Id}+\gamma \nabla g$ at $\prox_{\gamma g}(z^0)$. 

Applying the Inverse Function Theorem to \eqref{X} by setting $F$ as $\prox_{\gamma g}$ and $z^0$ as $p$ in \eqref{inversefunctionthm},  we have 
\[
J_{\prox_{\gamma g}}(z^0)=\Big(I+\gamma \nabla^2g\big(\prox_{\gamma g}(z^0)\big)\Big)^{-1},
\] 
Hence, Lemma \ref{Moreau} yields
\[
\nabla^2 g^{\gamma}(z^0)=\frac{1}{\gamma}\bigg(I-\Big(I+\gamma \nabla^2g\big(\prox_{\gamma g}(z^0)\big)\Big)^{-1}\bigg).
\]
According to \eqref{A}, 
\begin{align}
A(z^0)=A_1-\gamma A_2.\label{A2}
\end{align}
where 
\begin{align*}
    A_1&=2\Big(I+\gamma \nabla^2 g\big(\prox_{\gamma g}(z^0)\big)\Big)^{-1}-I,\\
    A_2&=L^T\nabla^2h\big(L\prox_{\gamma g}(z^0)\big)L\Big(I+\gamma \nabla^2 g\big(\prox_{\gamma g}(z^0)\big)\Big)^{-1}.\\
\end{align*}
Since $\gamma\in(0,\frac{1}{L_g})$, $A_1$ is invertible, as a result,
\[
\mathrm{det}\big(A(z^0)\big)=\mathrm{det}(A_1-\gamma A_2)=\mathrm{det}(I-\gamma A_2A_1^{-1})\mathrm{det}(A_1)=\prod_{i=1}^{n}(1-\gamma \lambda_i\big(A_2A_1^{-1})\big)\mathrm{det}(A_1),
\]
where $\lambda_i(A_2A_1^{-1}), i=1,...,n$ are the eigenvalues of $A_2A_1^{-1}$.

Let us set 
\[
C=I+\gamma\nabla^2g\big(\prox_{\gamma g}(z^0)\big)=C^T\succ 0,
\]
and rewrite $A_2A_1^{-1}$ as
\[
A_2A_1^{-1}=L^T\nabla^2 h\big(\prox_{\gamma g}(z^0)\big)LC^{-1}(2C^{-1}-I)^{-1}=L^T\nabla^2 h\big(\prox_{\gamma g}(z^0)\big)L(2I-C)^{-1}.
\]

Note that $L^T\nabla^2 h\big(\prox_{\gamma g}(z^0)\big)L$ is symmetric and $(2I-C)^{-1}$ is symmetric, positive definite. Therefore, $\lambda_i(A_2A_1^{-1})\in \mathbb{R}$, and we can set $\lambda_1(A_2A_1^{-1})\geq \lambda_2(A_2A_1^{-1}) ... \geq \lambda_n(A_2A_1^{-1})$.

In order to show $\mathrm{det}\big(A(z^0)\big)\neq 0$, it suffices to show that $1-\gamma \lambda_1>0$ when $\gamma\in(0,\frac{1}{L_g+L_h\|L\|_2^2})$. 

We have 
\begin{align*}
\lambda_1(A_2A_1^{-1})&\overset{(\mathrm{a})}{\leq}\lambda_1\Big(L^T\nabla^2h\big(\prox_{\gamma g}(z^0)\big)L\Big)\cdot\lambda_1\big((2I-C)^{-1}\big)\\
&\leq \lambda_1\Big(L^T\nabla^2h\big(\prox_{\gamma g}(z^0)\big)L\Big)\cdot \frac{1}{2-(1+\gamma L_g)}\\
&=\|(L^T\nabla^2h\big(\prox_{\gamma g}(z^0)\big)L\|_2\cdot \frac{1}{1-\gamma L_g}\\
&\overset{(\mathrm{b})}{\leq}\|L\|_2^2\|\nabla^2h\big(\prox_{\gamma g}(z^0)\big)\|_2\frac{1}{1-\gamma L_g}\\
&\leq\|L\|_2^2L_h\frac{1}{1-\gamma L_g},
\end{align*}
where (a) is by \cite[Corollary 11]{zhang2006eigenvalue}, and (b) is by Cauchy-Schwartz. Since $\gamma\in(0,\frac{1}{L_g+L_h\|L\|^2})$, we have \[
1-\gamma\lambda_1(A_2A_1^{-1})\geq 1-\|L\|_2^2L_h\frac{\gamma}{1-\gamma L_g}>0.
\]
Therefore, $\mathrm{det}\big(A(z^0)\big)\neq 0$.
\end{proof}
We can now write DYS iteration \eqref{iter} as gradient descent of an envelope under the following additional assumption.
\begin{assumption}
\label{assumption 2}
\hfill
\begin{enumerate}
    \item $f$ is $\beta_f-$weakly convex and $\gamma\in(0,\frac{1}{\beta_f})$.
    \item $g,h$ are twice continuously differentiable.
\end{enumerate}
\end{assumption}
\begin{theorem} 
\label{gradient descent}
Under Assumptions \ref{assumption 1} and \ref{assumption 2}, DYS iteration \eqref{iter} can be written equivalently as
\begin{equation}
    z^{k+1}=T(z^k)=z^k-\alpha\gamma A^{-1}(z^k)\nabla \varphi^{\gamma}(z^k),\label{iteration as gradient descent}
\end{equation}
where the metric and envelope are, respectively,
\begin{align}
A(z) &:=  I-2\gamma \nabla^2g^{\gamma}(z)-\gamma L^T\nabla^2h\big(L\prox_{\gamma g}(z)\big)L\big(I-\gamma \nabla ^2g^{\gamma}(z)\big), \nonumber\\
\varphi^{\gamma}(z)& :=  g^{\gamma}(z)-\gamma \|\nabla g^{\gamma}(z)\|^2-\gamma \langle q(z), \nabla g^{\gamma}(z)\rangle +h\big(L\prox_{\gamma g}(z)\big) \nonumber\\
&\quad -\frac{\gamma}{2}\|q(z)\|^2 +f^{\gamma}\Big(z-2\gamma \nabla g^{\gamma}(z)-\gamma q(z)\Big). \label{def1} 
\end{align}
\end{theorem}

\begin{proof}
In view of Lemma \ref{Moreau} and \eqref{w^k}, we have 
\begin{align}
w^k=p(z^k)-\prox_{\gamma g}(z^k)=\prox_{\gamma f}(r^k-\gamma q^k)-\prox_{\gamma g}(z^k),
\label{calculation}
\end{align}
where 
\begin{align*}
    \prox_{\gamma f}(r^k-\gamma q^k)&=r^k-\gamma q^k-\gamma \nabla f^{\gamma}(r^k-\gamma q^k).\\
    \prox_{\gamma g}(z^k)&=z^k-\gamma \nabla g^{\gamma}(z^k),\\
    r^k&=2\prox_{\gamma g}(z^k)-z^k=z^k-2\gamma \nabla g^{\gamma}(z^k),\\
    q^k&=q(z^k)=L^T\nabla h\Big(L\big(z^k-\gamma \nabla g^{\gamma}(z^k)\big)\Big).\\
\end{align*}
By substitution, 
\begin{align*}
    w^k&=-\gamma \nabla g^{\gamma}(z^k)-\gamma q(z^k) -\gamma \nabla f^{\gamma}\Big(z^k-2\gamma \nabla g^{\gamma}(z^k)-\gamma q(z^k)\Big).
\end{align*}
Let $\nabla_{z}$ denote taking gradient to $z$; then
\begin{align*}
    \nabla_z f^{\gamma}\Big(z-2\gamma \nabla g^{\gamma}(z)-\gamma q(z)\Big) 
    =A(z)\nabla f^{\gamma}\Big(z-2\gamma \nabla g^{\gamma}(z)-\gamma q(z)\Big),
\end{align*}
where $A(z)$ is given in \eqref{A}.
After some computation, we can verify that
\begin{align*}
    A(z^k)w^k=&-\gamma\big(\nabla_zg^{\gamma}(z^k)-\gamma \nabla_z\|\nabla g^{\gamma}(z^k)\|^2\big)-\gamma \Big(-\gamma\nabla_z\big(\langle q(z), \nabla g^{\gamma}(z^k)\rangle\big)\Big)\\
     &-\gamma \nabla_z h\big(L\prox_{\gamma g}(z)\big)-\gamma\big(-\frac{\gamma}{2}\nabla_z \|q(z)\|^2\big) -\gamma \nabla_z f^{\gamma}\Big(z^k-2\gamma \nabla g^{\gamma}(z^k)-\gamma q(z)\Big)\\
     &=-\gamma \nabla \varphi^{\gamma}(z^k).
\end{align*}
Since $A(z^k)$ is invertible, we can rewrite DYS iteration \eqref{iter} as \eqref{iteration as gradient descent}.
\end{proof}

\section{Properties of envelope}
\label{properties of the envelope}
In this section, we show that the global minimizers, local minimizers,  critical(stationary) points, and strict saddle points of the envelope $\varphi^{\gamma}$ defined in \eqref{def1} correspond \emph{one on one} to those of the objective function $\varphi$ in \eqref{objective}. 

First, we show lower and upper bounds of the envelope, which generalize \cite[Prop. 2.3]{themelis2016forward}, \cite[Prop. 4.3]{stella2017proximal}, and \cite[Prop. 1]{patrinos2014douglas}.
\begin{lemma}
\label{two inequalities}
Under Assumption \ref{assumption 1}, 
the following three inequalities hold for any $z\in\Rn$:
\begin{align}
    \varphi^{\gamma}(z)\leq & \, \varphi\big({\prox_{\gamma g}}(z)\big),  \label{inequ1}\\
    \varphi^{\gamma}(z) \geq &  \,\varphi\big(p(z)\big)
    +C_1(\gamma)\|p(z)-\prox_{\gamma g}(z)\|^2,
    \label{inequ2}\\
    \varphi^{\gamma}(z) \leq & \,\varphi\big(p(z)\big)
    +C_2(\gamma)\|p(z)-\prox_{\gamma g}(z)\|^2, \label{inequ3}
\end{align}
where $\varphi^{\gamma}(z)$ is defined in \eqref{def1}, \[C_1(\gamma)\coloneqq\frac{1-\gamma L_h\|L\|^2-\gamma L_g}{2\gamma}>0,\] 
\[C_2(\gamma)\coloneqq\frac{1+\gamma L_h\|L\|^2+\gamma L_g}{2\gamma}>0,\] 
and $p(z)$ is any element of\, $\prox_{\gamma f}\big(2\prox_{\gamma g}(z)-z-\gamma q(z)\big)$.
\end{lemma}

\begin{proof}[Proof of inequality \eqref{inequ1}] 
By applying Lemma \ref{Moreau} to $g$, $\varphi^{\gamma}(z)$ can be written as 
\begin{align}
    \varphi^{\gamma}(z)=&\min_{u}\{g(u)+\frac{1}{2\gamma}\|z-u\|^2\}-\gamma \|\frac{1}{\gamma}\big(z-\prox_{\gamma g}(z)\big)\|^2 \nonumber\\
    &-\gamma \langle q(z), \frac{1}{\gamma}\big(z-\prox_{\gamma g}(z)\big)\rangle \nonumber\\
    &+h\big(L\prox_{\gamma g}(z)\big)-\frac{\gamma}{2}\|q(z)\|^2 \nonumber\\
    &+\min_{u}\{f(u)+\frac{1}{2\gamma}\|-z+2\prox_{\gamma g}(z)-\gamma q(z)-u\|^2\}.
    \label{def2}
\end{align}
Taking $u=\prox_{\gamma g}(z)$ in the two minimums of \eqref{def2}, we have
\begin{align*}
    \varphi^{\gamma}(z)
    \leq&\, g\big(\prox_{\gamma g}(z)\big)+\frac{1}{2\gamma}\|z-\prox_{\gamma g}(z)\|^2-\gamma \|\frac{1}{\gamma}\big(z-\prox_{\gamma g}(z)\big)\|^2\\
    &-\langle q(z), z- \prox_{\gamma g}(z)\rangle\\
    &+h\big(L\prox_{\gamma g}(z)\big)-\frac{\gamma}{2}\|q(z)\|^2\\
    &+f\big(\prox_{\gamma g}(z)\big)+\frac{1}{2\gamma}\|-z+\prox_{\gamma g}(z)-\gamma q(z)\|^2\\
    =&\varphi\big(\prox_{\gamma g}(z)\big).
\end{align*}

\end{proof}
\begin{proof}[Proof of inequality \eqref{inequ2}] 
According to \cite[Thm. 1.25]{rockafellar2009variational}, we know that
$\prox_{\gamma f}\big(2\prox_{\gamma g}(z)-z-\gamma q(z)\big)\neq \varnothing$  for $\gamma\in(0,\frac{1}{L_g+L_h\|L\|^2})$.

By taking $u=\prox_{\gamma g}(z)$ in the first minimum of \eqref{def2} and $u=p(z)\in \prox_{\gamma f}(2\prox_{\gamma g}(z)-z-\gamma q(z))$ in the second, we have
\begin{align}
    \varphi^{\gamma}(z)=&\, g\big(\prox_{\gamma g}(z)\big)+\frac{1}{2\gamma}\|z-\prox_{\gamma g}(z)\|^2 \nonumber\\
    &-\gamma \langle q(z), \frac{1}{\gamma}\big(z-\prox_{\gamma g}(z)\big)\rangle \nonumber\\
    &+h\big(L\prox_{\gamma g}(z)\big)-\frac{\gamma}{2}\|q(z)\|^2 \nonumber\\
    &+f\big(p(z)\big)+\frac{1}{2\gamma}\|-z+2\prox_{\gamma g}(z)-\gamma q(z)-p(z)\|^2.
    \label{def3}
\end{align}
By making use of 
\[
h(y)\geq h(x)-\langle \nabla h(y),x-y\rangle-\frac{L_h}{2} \|x-y\|^2\,\,\, \mathrm{for any}\,\, x,y\in \Rm,
\] 
we arrive at
\begin{align*}
\varphi^{\gamma}(z)
    \geq&\,g\big(\prox_{\gamma g}(z)\big)-\frac{1}{2\gamma}\|z-\prox_{\gamma g}(z)\|^2\\
    &-\langle q(z), z-\prox_{P\gamma g}(z)\rangle\\
    &+h\big(Lp(z)\big)-\langle \nabla h\big(L\prox_{\gamma g}(z)\big), L(p(z)-\prox_{\gamma g}(z)) \rangle\\
    &-\frac{L_h}{2}\|L\big(p(z)-\prox_{\gamma g}(z)\big)\|^2-\frac{\gamma}{2}\|q(z)\|^2\\
    &+f\big(p(z)\big)+\frac{1}{2\gamma}\|2\prox_{\gamma g}(z)-z-\gamma q(z)-p(z)\|^2.
\end{align*}
Next, by making use of $\|a+b+c\|^2=\|a\|^2+\|b^2\|+\|c\|^2+2\langle a, b\rangle+2\langle b, c\rangle+2\langle a, c\rangle$ for
\begin{align*}
a&=\prox_{\gamma g}(z)-p(z),\\
b&=\prox_{\gamma g}(z)-z,\\
c&=-\gamma q(z),
\end{align*}
we obtain
\begin{align*}
    \varphi^{\gamma}(z)\geq &
    g\big(\prox_{\gamma g}(z)\big)+h\big(Lp(z)\big)-\frac{L_h}{2}\|L\big(p(z)-\prox_{\gamma g}(z)\big)\|^2\\
    &+f\big(p(z)\big)+\frac{1}{2\gamma}\|\prox_{\gamma g}(z)-p(z)\|^2+\langle \prox_{\gamma g}(z)-p(z), \frac{1}{\gamma}(\prox_{\gamma g}(z)-z)\rangle.
\end{align*}
Finally, by substituting
\begin{align*}
    \nabla g\big(\prox_{\gamma g}(z)\big)&=-\frac{1}{\gamma}(\prox_{\gamma g}(z)-z),\\
    g(y)&\geq g(x)-\langle \nabla g(y),x-y\rangle-\frac{L_g}{2} \|x-y\|^2\,\,\,\ \mathrm{for any}\,\, x,y\in \Rn,
\end{align*}
we arrive at \eqref{inequ2}.
\end{proof}
\begin{proof}[Proof of inequality \eqref{inequ3}]
Similarly to the proof above, we can also start from \eqref{def3} and apply \[h(y)\leq h(x)-\langle \nabla h(y),x-y\rangle+\frac{L_g}{2} \|x-y\|^2 \,\,\,\mathrm{for any}\,\, x,y\in \Rm,\]  
\[g(y)\leq g(x)-\langle \nabla g(y),x-y\rangle+\frac{L_g}{2} \|x-y\|^2 \,\,\,\mathrm{ for any} \,\,x,y\in \Rn,\]
which gives \eqref{inequ3}.
\end{proof}

Now we can establish the direct connections between the global and local minimizers of $\varphi^{\gamma}$ and those of $\varphi$. These results generalize \cite[Prop. 2.3]{themelis2016forward} and \cite[Thm. 4.4]{stella2017proximal}.

\begin{theorem}
\label{globalmin}
Under Assumption \ref{assumption 1}, 
we have
\begin{enumerate}
    \item $\inf_{x\in \Rn}\varphi(x)=\inf_{z\in\Rn}\varphi^{\gamma}(z)$,
    \item $\argmin_{x\in\Rn}\varphi(x)=\prox_{\gamma g}\Big(\argmin_{z\in\Rn}\big(\varphi^{\gamma}(z)\big)\Big)$.
\end{enumerate}
\end{theorem}

\begin{proof}[Proof of 1]
From \eqref{inequ1} we have
\[
\inf_{z\in\Rn} \varphi^{\gamma}(z)\leq \inf_{x\in \Rn}\varphi(x), 
\]
If $\inf_{z\in\Rn} \varphi^{\gamma}(z)< \inf_{x\in \Rn}\varphi(x) $, then there exists $z_1\in\Rn$ such that $\varphi^{\gamma}(z_1)<\inf_{x\in \Rn}\varphi(x)$.  Then \eqref{inequ2} gives
\[
\inf_{x\in\Rn} \varphi(x)>\varphi^{\gamma}(z_1)\geq \varphi\big(p(z_1)\big)+C(\gamma)\|\prox_{\gamma g}(z_1)-p(z_1)\|^2\geq \varphi\big(p(z_1)\big),
\]
which is a contradiction.
\end{proof}

\begin{proof}[Proof of 2]
Let us first show that 
\[
\argmin_{x\in\Rn}\varphi(x)\subseteq\prox_{\gamma g}\Big(\argmin_{z\in\Rn}\big(\varphi^{\gamma}(z)\big)\Big).
\]
There is nothing to show if $\argmin_{x\in\Rn}\varphi(x)=\varnothing$; If $\argmin_{x\in\Rn}\varphi(x)\neq\varnothing$, then,  for any $x^*\in \argmin_{x\in\Rn}\varphi(x)$, we have
$x^*=\prox_{\gamma g}(z^*)$ for $z^*=(I+\gamma \nabla g)(x^*)$. As a result, \eqref{inequ1} and \eqref{inequ2} give us
\[
\inf_{x\in\Rn}\varphi(x)=\varphi(x^*)=\varphi\big(\prox_{\gamma g}(z^*)\big)\geq \varphi^{\gamma}(z^*)\geq \varphi\big(p(z^*)\big)+C(\gamma)\|\prox_{\gamma g}(z^*)-p(z^*)\|^2.
\]
Which enforces $\prox_{\gamma g}(z^*)=p(z^*)$ and $\varphi\big(\prox_{\gamma g}(z^*)\big)=\varphi^{\gamma}(z^*)$. So for any $z\in\Rn$ we have
\[
\varphi^{\gamma}(z^*)=\inf_{x\in\Rn}\varphi(x)\leq\varphi\big(p(z)\big)\leq \varphi^{\gamma}(z)-C(\gamma)\|\prox_{\gamma g}(z)-p(z)\|^2\leq \varphi^{\gamma}(z),
\]
which yields $z^*\in\argmin_{z\in\Rn}\varphi^{\gamma}(z), x^*\in\prox_{\gamma g}\Big(\argmin_{z\in\Rn}\big(\varphi^{\gamma}(z)\big)\Big)$.

Now let us show that 
\[
\prox_{\gamma g}\Big(\argmin_{z\in\Rn}\big(\varphi^{\gamma}(z)\big)\Big)\subseteq\argmin_{x\in\Rn}\varphi(x).
\]
Again, we can assume that
$\argmin_{z\in\Rn}\big(\varphi^{\gamma}(z)\big)\neq\varnothing$. For any $z^*\in \argmin_{z\in\Rn} \varphi^{\gamma}(z)$, we need to show $\prox_{\gamma g}(z^*)\in \argmin_{x\in\Rn}\varphi(x)$.

Let $z^{**}=(I+\gamma \nabla g)p(z^*)$, 
then $\prox_{\gamma g}(z^{**})=p(z^*)$ and \eqref{inequ1} and \eqref{inequ2} give us
\[
\varphi^{\gamma}(z^{**})\leq \varphi(\prox_{\gamma g}\big(z^{**})\big)=\varphi\big(p(z^*)\big)\leq \varphi^{\gamma}(z^*)-C(\gamma)\|\prox_{\gamma g}(z^*)-p(z^*)\|^2.
\]
Since $z^*\in \argmin_{z\in\Rn} \varphi^{\gamma}(z)$, we must have \[
\prox_{\gamma g}(z^*)=p(z^*)=\prox_{\gamma g}(z^{**}),
\]
\[
\varphi^{\gamma}(z^{*})=\varphi^{\gamma}(z^{**})=\varphi\big(\prox_{\gamma g}(z^*)\big).
\]
Consequently, for any $z\in\Rn$ we have
\[
\varphi\big(\prox_{\gamma g}(z^*)\big)=\varphi^{\gamma}(z^{*})\leq\varphi^{\gamma}(z)\leq \varphi\big(\prox_{\gamma g}(z)\big),
\]
which concludes $\prox_{\gamma g}(z^*)\in \argmin_{x\in\Rn}\varphi(x).$
\end{proof}

It turns out that the local minimizers of $\varphi$ and $\varphi^{\gamma}$ also have a one-to-one correspondence.

\begin{theorem}
\label{localmin}
Under Assumptions \ref{assumption 1} and \ref{assumption 2}, 
we have:
\begin{enumerate}
    \item If $\prox_{\gamma g}(z^*)\in \argmin_{x\in B(\prox_{\gamma g}(z^*),\delta)}\varphi(x)$ for some $\delta>0$, 
    then 
    $z^*$ is a local minimizer of $\varphi^{\gamma}$.
    \item If $z^*\in \argmin_{z\in B(z^*,\varepsilon)}\varphi^{\gamma}(z)$ for some $\varepsilon>0$, then 
    \[
    \varphi\big(\prox_{\gamma g}(z^*)\big)\leq \varphi\big(\prox_{\gamma g}(z)\big)\,\,\, \text{for all $z$ such that }\,\,\, \|z-z^*\|\leq \varepsilon.
    \]
    That is, $\prox_{\gamma g}(z^*)$ is a local minimizer of $\varphi(x)$.
\end{enumerate}
\end{theorem}

\begin{proof}[Proof of 1]
Since $\prox_{\gamma g}(z^*)$ is a local minimizer of $\varphi$, according to \cite[Exercise 10.10]{rockafellar2009variational}, we have 
\[
\mathbf{0}\in \partial \varphi\big(\prox_{\gamma g}(z^*)\big)= \partial f\big(\prox_{\gamma g}(z^*)\big)+\nabla g\big(\prox_{\gamma g}(z^*)\big)+q\big(z^{*})\big).
\]
Since $\prox_{\gamma g}$ is single-valued, this is equivalent to
\[
\mathbf{0}\in \partial f\big(\prox_{\gamma g}(z^*)\big)+\frac{1}{\gamma}(-\prox_{\gamma g}(z^*)+z^*+\gamma q\big(z^{*})\big),
\]
Since $f+\frac{1}{2\gamma}\|\cdot\|^2$ is convex and $\prox_{\gamma f}$ is single valued, this is further equivalent to
\[
\prox_{\gamma g}(z^*)=\prox_{\gamma f}\big(2\prox_{\gamma g}(z^*)-z^*-\gamma q(z^{*})\big)=p(z^*).
\]
According to Lemma \ref{Moreau}, $\prox_{\gamma f}$ is $\frac{1}{1-\gamma \beta_f}-$Lipschitz continuous, we can conclude that there exists $\eta>0$ such that when $\|z-z^*\|\leq \eta$, we have $\|p(z)-p(z^*)\|\leq \delta$ and
\[
\varphi^{\gamma}(z^*)=\varphi\big(\prox_{\gamma g}(z^*)\big)=\varphi\big(p(z^*)\big)\leq \varphi\big(p(z)\big)\leq \varphi^{\gamma}(z)-C(\gamma)\|\prox_{\gamma g}(z)-p(z)\|^2\leq \varphi^{\gamma}(z).
\]
\end{proof}

\begin{proof}[Proof of 2]
According to Lemma \ref{g related differentiablility}, $A(z)$ is invertible at $z^*$. Theorem \ref{gradient descent} tells us that $\varphi^{\gamma}$ is differentiable at $z^*$, so $\nabla \varphi^{\gamma}(z^*)=\mathbf{0}$ and $\prox_{\gamma g}(z^*)=p(z^*)$. As a result, for any $z\in\Rn$ with $\|z-z^*\|\leq \varepsilon$ we have
\[
\varphi\big(\prox_{\gamma g}(z^*)\big)=\varphi^{\gamma}(z^*)\leq \varphi^{\gamma
}(z)\leq \varphi\big(\prox_{\gamma g}(z)\big).
\]
Furthermore, according to Lemma \ref{g related differentiablility} we have
\[
\prox_{\gamma g}(z)= \prox_{\gamma g}(z^*)+\Big(I+\gamma \nabla^2g\big(\prox_{\gamma g}(z^*)\big)\Big)^{-1}(z-z^*)+o(\|z-z^*\|).
\]
Since $\Big(I+\gamma \nabla^2g\big(\prox_{\gamma g}(z^*)\big)\Big)^{-1}$ is positive definite, we can conclude that $\prox_{\gamma g}\big(B(z^*,\varepsilon)\big)$ contains a ball centered at $\prox_{\gamma g}(z^*)$, as a result, $\prox_{\gamma g}(z^*)$ is a local minimizer of $\varphi(x)$.
\end{proof}

Now, let us show the one-to-one correspondence between the critical points of the envelope $\varphi^{\gamma}$ and the stationary points of the objective $\varphi(x)$.

\begin{theorem}
\label{critical point and stationary point}
Under Assumptions \ref{assumption 1} and \ref{assumption 2}, $z^*$ is a critical point of $\varphi^{\gamma}$ if and only if $\prox_{\gamma g}(z^*)$ is a stationary point of $\varphi$.
\end{theorem}
\begin{proof}
Since $f$ is $\beta_f-$weakly convex and $\gamma\in(0,\frac{1}{\beta_f})$, by Lemma \ref{Moreau}, we know that $\prox_{\gamma f}$ is single-valued. And by Theorem \ref{gradient descent}, we have
\begin{equation}
\label{first order derivative}
\nabla \varphi^{\gamma}(z)=-A(z)\frac{1}{\gamma}(p(z)-\prox_{\gamma g}(z)),
\end{equation}
where $p(z)=\prox_{\gamma f}\Big((2\prox_{\gamma g}(z)-z-\gamma L^T\nabla h\big(L\prox_{\gamma g}(z)\big)\Big)$. So $\nabla \varphi^{\gamma}(z^*)=\mathbf{0}$ if and only if
\begin{align*}
\prox_{\gamma g}(z^*)&=\prox_{\gamma f}\Big(2\prox_{\gamma g}(z^*)-z^*-\gamma L^T\nabla h(L\prox_{\gamma g}(z^*)\big)\Big)\nonumber\\
&=\argmin_{z}\{f(z)+\frac{1}{2\gamma}\|z-\Big(2\prox_{\gamma g}(z^*)-z^*-\gamma L^T\nabla h\big(L\prox_{\gamma g}(z^*)\big)\Big)\|^2\}.
\end{align*}
Since the objective in the $\argmin$ is convex,  by \cite[Exercise 10.10]{rockafellar2009variational} we know that this is equivalent to 
\[
\mathbf{0}\in \partial f\big(\prox_{\gamma g}(z^*)\big)+\frac{1}{\gamma}\Big(-\prox_{\gamma g}(z^*)+z^*+\gamma L^T\nabla h(L\prox_{\gamma g}(z^*)\big)\Big).
\]
By the definition of $\prox_{\gamma g}$ and $\gamma\in(0,\frac{1}{L_g+L_h\|L\|^2})$, this is further equivalent to
\[
\mathbf{0}\in \partial f\big(\prox_{\gamma g}(z^*)\big)+\nabla g\big(\prox_{\gamma g}(z^*)\big)+L^T\nabla h\big(L\prox_{\gamma g}(z^*)\big)=\partial \varphi\big(\prox_{\gamma g}(z^*)\big).
\]
\end{proof}

In order to establish the correspondence between the strict saddles of $\varphi^{\gamma}$ and $\varphi$, we also need the following assumption.
\begin{assumption}
\label{assumption 3}
    For any critical point $z^*$ of $\varphi^{\gamma}$, $f$ is twice continuously differentiable in a small neighbourhood of $\prox_{\gamma g}(z^*)$, and there exits $L_f>0$ such that  $\|\nabla^2 f(\prox_{\gamma g}(z^*))\|\leq L_f$. In addition, $\gamma\in(0,\frac{1}{L_f})$.
\end{assumption}

\begin{lemma}
\label{twice differentiability of envelope}
Let $z^*$ be a critical point of $\varphi^{\gamma}$. Under Assumptions \ref{assumption 1}, \ref{assumption 2} and \ref{assumption 3}, 
$\varphi^{\gamma}$ is twice differentiable at $z^*$ and
\begin{align}
\nabla^2\varphi^{\gamma}(z^*)&=-A(z^*)\frac{1}{\gamma}\big(J_p(z^*)-J_{\prox_{\gamma g}}(z^*)\big)\label{second order derivative}\\
&=-A(z^*)\frac{1}{\gamma}\Big(I+\gamma \nabla^2 f\big(p(z)\big)\Big)^{-1}A(z^*)+\frac{1}{\gamma}A(z^*)\Big(I+\gamma \nabla^2 g\big(\prox_{\gamma g}(z^*)\big)\Big)^{-1}.\label{expression of second order}
\end{align}
In case of $h=0$ (DRS) and $g=0$ (FBS), $\nabla^2\varphi^{\gamma}(z^*)$ is symmetric.
\end{lemma}
\begin{proof}
\eqref{second order derivative} follows from \eqref{first order derivative}, $p(z^*)=\prox_{\gamma g}(z^*)$, and \cite[Prop. 2.A.2]{stella2017proximal}, \eqref{expression of second order} follows from Lemma \ref{g related differentiablility} and chain rule.

When $g=0$ or $h=0$, \eqref{A} tells us that $A(z^*)$ is symmetric, so the first part on the right hand side of \eqref{expression of second order} is symmetric. 

When $h=0$, we have
\[
A(z^*)\Big(I+\gamma \nabla^2 g\big(\prox_{\gamma g}(z^*)\big)\Big)^{-1}=\Big(2\Big(I+\gamma \nabla^2 g\big(\prox_{\gamma g}(z^*)\big)\Big)^{-1}-I\Big)\Big(I+\gamma \nabla^2 g\big(\prox_{\gamma g}(z^*)\big)\Big)^{-1},
\]
so the second part is also symmetric. When $g=0$, the second part is $\frac{1}{\gamma}A(z^*)$, thus symmetric.

So we can conclude that when $h=0$ or $g=0$, $\nabla^2\varphi^{\gamma}(z^*)$ is symmetric.
\end{proof}

\begin{theorem}
\label{strict saddle} 
Let $z^*$ be a critical point of $\varphi^{\gamma}$. Under Assumptions \ref{assumption 1},  \ref{assumption 2} and \ref{assumption 3}, when $h=0$ (DRS) or $g=0$ (FBS), $z^*$ is a strict saddle point of $\varphi^{\gamma}$ if and only if $\prox_{\gamma g}(z^*)$ is a strict saddle point of $\varphi$.
\end{theorem}
\begin{proof}
According to Lemma \ref{twice differentiability of envelope}, we know that $\nabla^2 \varphi^{\gamma}(z^*)$ exists and it is symmetric.

Let $z^*$ be a strict saddle point of $\varphi^{\gamma}(z)$, then Taylor expansion gives
\begin{align*}
    \varphi^{\gamma}(z)&=\varphi^{\gamma}(z^*)+\frac{1}{2}(z-z^*)^T\nabla^2\varphi^{\gamma}(z^*)(z-z^*)+o(\|z-z^*\|^2),\\
    \varphi\big(p(z)\big)&=\varphi\big(p(z^*)\big)+\frac{1}{2}\big(p(z)-p(z^*)\big)^T\nabla^2 \varphi\big(p(z^*)\big)\big(p(z)-p(z^*)\big)+o(\|p(z)-p(z^*)\|^2).
\end{align*}
On the other hand, \eqref{inequ2} gives
\[
\varphi^{\gamma}(z)\geq \varphi\big(p(z)\big).
\]
Let $\nabla^2\varphi^{\gamma}(z^*)v=\lambda v$, where $\|v\|=1$ and $\lambda<0$. Setting $z-z^*=\alpha v$, we arrive at
\begin{align}
    &\varphi^{\gamma}(z^*)+\frac{1}{2}\lambda\alpha^2+o(\alpha^2)\nonumber\\
    &\geq\varphi\big(p(z^*)\big)+\frac{1}{2}\big(p(z)-p(z^*)\big)^T\nabla^2 \varphi\big(p(z^*)\big)\big(p(z)-p(z^*)\big)+o(\|p(z)-p(z^*)\|^2).\label{connect}
\end{align}
Furthermore, \eqref{inequ1}, \eqref{inequ2} together with $p(z^*)=\prox_{\gamma g}(z^*)$ yield $\varphi^{\gamma}(z^*)=\varphi\big(p(z^*)\big)$, combine this with \eqref{connect} and $p(z)-p(z^*)=O(\|z-z^*\|)=O(\alpha)$, we conclude that $\lambda_{\min}\Big(\nabla^2\varphi\big(\prox_{\gamma g}(z^*)\big)\Big)<0$.

Similarly, let $\prox_{\gamma g}(z^*)$ be a strict saddle of $\varphi(z)$, then Taylor expansions gives
\begin{align*}
    \varphi^{\gamma}(z)&=\varphi^{\gamma}(z^*)+\frac{1}{2}(z-z^*)^T\nabla^2\varphi^{\gamma}(z^*)(z-z^*)+o(\|z-z^*\|^2),\\
    \varphi(\prox_{\gamma g}\big(z)\big)&=\varphi\big(\prox_{\gamma g}(z^*)\big)+\frac{1}{2}\big(\prox_{\gamma g }(z)-\prox_{\gamma g}(z^*)\big)^T\nabla^2 \varphi\big(\prox_{\gamma g}(z^*)\big)\big(\prox_{\gamma g }(z)-\prox_{\gamma g}(z^*)\big)\\
    &+o(\|\prox_{\gamma g }(z)-\prox_{\gamma g}(z^*)\|^2).
\end{align*}
On the other hand, \eqref{inequ1} gives
\[
\varphi^{\gamma}(z)\leq \varphi\big(\prox_{\gamma g}(z)\big),
\]
Let $\nabla^2\varphi\big(\prox_{\gamma g}(z^*)\big)v=\lambda v$ where $\lambda<0$ and $\|v\|=1$. By setting $z=(\mathrm{Id}+\gamma \nabla g)(\prox_{\gamma g}(z^*)+\alpha v)$, we obtain $\prox_{\gamma g}(z)-\prox_{\gamma g}(z^*)=\alpha v$, taking $\alpha\rightarrow 0$ gives $\lambda_{\min}\big(\nabla^2\varphi^{\gamma}(z^*)\big)<0$.
\end{proof}

\section{Avoidance of strict saddle points}
\label{avoidance}

In this section, we first show that under smoothness conditions, the probability for DRS and FBS with random initializations to converge to strict saddle points of DRE and FBE is zero, respectively. Then, by combining this result with the correspondence between the strict saddle points of the envelope and the objective, as stated in Theorem \ref{strict saddle}, we can conclude that DRS and FBS, if converge, will almost always avoid the strict saddle points of the objective.
Therefore, when the objective satisfies the ``strict saddle property", DRS and FBS, if they converge, will almost always converge to local minimizers.

In order to prove the main result, Theorem \ref{to phi}, we need the following Stable-Center Manifold Theorem, and its direct consequence, Theorem \ref{probability is zero}.

Theorem \ref{stable manifold Thm} states that, if $T$ is a local diffeomorphism around one of its fixed point $z^*$, then there is a local stable center manifold $W^{\mathrm{cs}}_{\mathrm{loc}}$ with dimension equal to the number of eigenvalues of the Jacobian of $T$ at $z^*$ that are less than or equal to $1$. Furthermore, there exists a neighbourhood $B$ of $z^*$, such that a point $z$ must be in $W^{\mathrm{cs}}_{\mathrm{loc}}$ if its forward iterations $T^k(z)$, for all $k\geq 0$, stay in $B$.

\begin{theorem}[Theorem III.7, Shub \cite{shub2013global}]
\label{stable manifold Thm}
Let $z^*$ be a fixed point for a $C^{r}$ local diffeomorphism $T: U\rightarrow\Rn$, where $U$ is a neighbourhood of $z^*$ and $r\geq 1$. Suppose $E=E_s\bigoplus E_{u}$, where $E_s$ is the span of the eigenvectors that correspond to eigenvalues of $J_T (z^*)$ that have magnitude less than, or equal to, $1$, and $E_u$ is the span of eigenvectors that correspond to eigenvalues of $J_T (z^*)$ that have magnitude greater than $1$. Then there exists a $C^r$ embedded disk $W^{\mathrm{cs}}_{\mathrm{loc}}$ that is tangent to $E_s$ at $z^*$, which is called the local stable center manifold. Moreover, there exists a neighbourhood $B$ of $z^*$, such that $T(W^{\mathrm{cs}}_{\mathrm{loc}})\cap B\subseteq W^{\mathrm{cs}}_{\mathrm{loc}}$, $\cap_{k=0}^{\infty} T^{-k}(B)\subseteq W^{\mathrm{cs}}_{\mathrm{loc}}$, where $T^{-k}(B)=\{z\in\Rn\,|\,T^k(z)\in B\}$.
\end{theorem}

The assumption of this following theorem is weaker than that of Theorem 2 of \cite{lee2017first}, as we do not assume that $T$ is invertible in $\Rn$ but only around every $z^*\in A^*_{T}$.
\begin{theorem}
\label{probability is zero}
Assume that $T(z)=z+\alpha \big(p(z)-\prox_{\gamma g}(z)\big)$ is a local diffeomorphism around every $z^*\in A^*_{T}$, where $A^*_{T}=\{z\in\Rn\,|\, z=T(z), \max_{i}\lambda_i\big(J_T(z)\big)>1\}$ is the set of unstable fixed points of $T$. 
Then the set $W=\{z^0:\,\lim z^k \in A^*_T\}$ has Lebesgue measure $\mu(W)=0$ in $\Rn$.
\end{theorem}

\begin{proof}

Take any $z^0\in W$, we have $z^k=T^k(z^0)\rightarrow z^*\in A^*_{T}$, there exists $t_0>0$, such that for any $t\geq t_0$ we have $T^t(z^0)\in B_{z^*}$. So $T^t(z^0)\in S\triangleq \cap_{k=0}^{\infty} T^{-k}(B_{z^*})$ for any $t\geq t_0$.

From Theorem \ref{stable manifold Thm} we know that 
$S$ is a subset of the local center stable manifold $W^{\mathrm{cs}}_{\mathrm{loc}}$ whose codimension is greater or equal to $1$, so we have $\mu(S)=0$;

Finally, $T^{t_0}(z^0)\in S$ implies that $z^0\in T^{-t_0}(S)\subseteq \cup_{j=0}^{\infty} T^{-j}(S)$, since
\[
T^{-j}(S)=T^{-j}\big(\cap_{k=0}^{\infty} T^{-k}(B_{z^*})\big)= \cap_{k=j}^{\infty} T^{-k}(B_{z^*})\subseteq \cap_{k=0}^{\infty} T^{-k}(B_{z^*})=S,
\]
we can conclude that $\mu(W)=0$.
\end{proof}

Now let us show that $T(z)$ in Theorem \ref{probability is zero} is indeed a local diffeomorphism around its fixed points.

\begin{lemma}
\label{local differomorphism}
Let $T(z)=z+\alpha \big(p(z)-\prox_{\gamma g}(z)\big)$ and $z^*\in\mathrm{Fix}T$. Under Assumptions \ref{assumption 1}, \ref{assumption 2} and \ref{assumption 3}, $T$ is a local diffeomorphism around $z^*$ in the following two cases.
\begin{enumerate}
    \item $h=0$ and $\alpha\in(0,\alpha_1)$, where
    \[
    \alpha_1=\frac{2}{1-\frac{1-\gamma L_g}{1+\gamma L_g}\frac{1-\gamma L_f}{1+\gamma L_f}}.
    \]
    \item $g=0$ and $\alpha\in(0,\alpha_2)$, where
    \[
    \alpha_2=\frac{1+\gamma L_f}{\gamma(L_h\|L\|^2+L_f)}.
    \]
\end{enumerate}
\end{lemma}
\begin{proof}
By Assumption \ref{assumption 1} and Lemma \ref{Moreau}, $p(z)$ is continuous, therefore when $z$ sufficiently close to $z^*$, $p(z)$ is in the neighbourhood of $\prox_{\gamma g}(z^*)$ guaranteed by Assumption \ref{assumption 3}. Lemma \ref{g related differentiablility} and chain rule tell us that 
\begin{align*}
    J_{p}(z)&=\Big(I+\gamma \nabla^2 f\big(p(z)\big)\Big)^{-1}A(z),\\
    J_{\prox_{\gamma g}}(z)&=\Big(I+\gamma \nabla^2 g\big(\prox_{\gamma g}(z)\big)\Big)^{-1},
\end{align*}
where $A(z)$ is defined in \eqref{A}, therefore $J_T(z)$ exists and $J_T(z)=I+\alpha (J_p(z)-J_{\prox_{\gamma g}}(z))$. 

For the local invertibility of $T$ around $z^*$ , let us show that $\lambda_{\mathrm{min}}(J_T(z))>0$ for $z$ sufficiently close to $z^*$. 
\begin{enumerate}
\item When $h=0$ and $\alpha\in(0,\alpha_1)$, 
let 
\begin{align*}
    B_1(z)&=\Big(I+\gamma \nabla^2g\big(\prox_{\gamma g}(z)\big)\Big)^{-1},\\
    B_2(z)&=\Big(I+\gamma \nabla^2 f\big(p(z)\big)\Big)^{-1}.
\end{align*}
Then from \eqref{A} we have
\begin{align*}
   J_T(z)&=I+\alpha\big(B_2(z)A(z)-B_1(z)\big)\\
    &=I+\alpha\big(B_2(z)(2B_1(z)-I)-B_1(z)\big)\\
    &=I-\frac{1}{2}\alpha I+\alpha \big(B_2(z)-\frac{1}{2}I\big)\big(2B_1(z)-I\big).
\end{align*}
Since $2B_1(z)-I$ is positive definite,  \cite[Corollary 11]{zhang2006eigenvalue} tells us that
\[
\lambda_{\mathrm{min}}(J_T(z))\geq 1-\frac{1}{2}\alpha +\alpha(\frac{1}{1+\gamma L_f}-\frac{1}{2})(\frac{2}{1+\gamma L_g}-1).
\]
As a result, $\lambda_{\mathrm{min}}(J_T(z))> 0$ when $\alpha\in(0,\alpha_1)$.

\item When $g=0$ and $\alpha\in(0,\alpha_2)$, let 
\[
B_3(z)=L^T\nabla^2 h(L\prox_{\gamma g}(z))L,
\]
then
\begin{align*}
    J_T(z)&=I+\alpha\big(B_2(z)A(z)-B_1(z)\big)\\
    &=I+\alpha\Big(B_2(z)\big(I-\gamma B_3(z)\big)-I\Big)\\
    &=I-\alpha I+\alpha B_2(z)\big(I-\gamma B_3(z)\big).
\end{align*}
By Assumption \ref{assumption 1}, $I-\gamma B_3(z)$ is positive definite. \cite[Corollary 11]{zhang2006eigenvalue} gives
\[
\lambda_{\mathrm{min}}(J_T(z))\geq 1-\alpha+\alpha (1-\gamma\|L\|^2L_h)\frac{1}{1+\gamma L_f}.
\]
As a result, $\lambda_{\mathrm{min}}(J_T(z))>0$ when $\alpha\in(0,\alpha_2)$.
\end{enumerate}
\qed
\end{proof}

Now we are ready to show the main result of this section: when $\gamma$ and $\alpha$ are small enough, it is impossible for DRS and FBS to converge to any strict saddle point of $\varphi^{\gamma}$, thus any strict saddle point of $\varphi$.

\begin{lemma}
\label{fixed point and stationary point}
Under Assumptions \ref{assumption 1} and \ref{assumption 2}, $z^*\in\mathrm{Fix}T$ if and only if $\nabla \varphi^{\gamma}(z^*)=\mathbf{0}$.
\end{lemma}
\begin{proof}
This follows directly from Theorem \ref{gradient descent}.
\end{proof}

\begin{theorem}
\label{final result}
Define $Z^*=\{z^*\in\mathbb{R}^n\,|\, \nabla \varphi^{\gamma}(z^*)=0, \lambda_{\min}\big(\nabla^2\varphi^{\gamma}(z^*)\big)<0 \}$ as the set of strict saddle points of $\varphi^{\gamma}$. Under Assumptions \ref{assumption 1}, \ref{assumption 2}, and \ref{assumption 3}, then in each of the following cases, 
\begin{enumerate}
    \item $h=0$ (DRS) and $\alpha\in(0,\alpha_1)$, where
    \[
    \alpha_1=\frac{2}{1-\frac{1-\gamma L_g}{1+\gamma L_g}\frac{1-\gamma L_f}{1+\gamma L_f}}.
    \]
    \item $g=0$ (FBS) and $\alpha\in(0,\alpha_2)$, where
    \[
    \alpha_2=\frac{1+\gamma L_f}{\gamma(L_h\|L\|^2+L_f)}.
    \]
\end{enumerate}
the set $W=\{z^0\in\Rn\,|\,\lim z^k \in Z^*\}$ satisfies $\mu(W)=0$.
\end{theorem}
\begin{proof}
Take any $z^*\in Z^*$, Lemma \ref{g related differentiablility} states that $A(z^*)$ is invertible, and $A(z)$ defined in \eqref{A} is symmetric when $h=0$ or $g=0$. Also, $\nabla^2\varphi(z^*)$ is symmetric when $h=0$ or $g=0$. 

According to \eqref{second order derivative}, we have
\[
J_p(z^*)-J_{\prox_{\gamma g}}(z^*)=-\gamma A^{-1}(z^*)\nabla^2\varphi^{\gamma}(z^*),
\]
since $A^{-1}(z^*)\nabla^2\varphi^{\gamma}(z^*)$ is similar to $A^{-\frac{1}{2}}(z^*)\nabla^2\varphi^{\gamma}(z^*)A^{-\frac{1}{2}}(z^*)$, we know that $J_p(z^*)-J_{\prox_{\gamma g}}(z^*)$ has real eigenvalues and  
\[
\lambda_{\mathrm{max}}\big(J_p(z^*)-J_{ \prox_{\gamma g}}(z^*)\big)>0.
\]
Since 
\[
\lambda_{\mathrm{max}}\big(J_ T (z^*)\big)=1+\alpha\lambda_{\mathrm{max}}\big(J_  p(z^*)-J_{\prox_{\gamma g}}(z^*)\big),
\]
from Theorem \ref{critical point and stationary point}, \ref{strict saddle}, and Lemma \ref{fixed point and stationary point}, we know that $Z^*=A^*_{T}$. 

Furthermore, from Lemma \ref{local differomorphism} we know that $T$ is a local diffeomorphism around every $z^*\in Z^*=A^*_T$, therefore Theorem \ref{probability is zero} gives $\mu(W)=0$.
\qed
\end{proof}

\begin{theorem}
\label{to phi}
Define $X^*=\{x^*\in\Rn\,|\,\nabla \varphi(x^*)=0, \lambda_{\min}(\nabla^2 \varphi(x^*))<0\}$. Under Assumptions \ref{assumption 1}, \ref{assumption 2}, and \ref{assumption 3},
the set $V=\{z^0\in\Rn\,|\,\lim \prox_{\gamma g}(z^k)\in X^* \}$ satisfies $\mu(V)=0$.
\begin{proof}
Combine Theorem \ref{strict saddle} with Theorem \ref{final result}.
\end{proof}
\end{theorem}

By Theorem \ref{to phi}, under smoothness conditions, DRS and FBS iterates will almost always avoid the strict saddle points of the objective. When the objective satisfies the \textit{strict saddle property}, i.e., the saddle points of the objective are either local minimizers or strict saddle points, we can conclude that FBS and DRS almost always converge to local minimizers of the objective whenever they converge.

Many problems in practice satisfy the strict saddle property. Examples include dictionary learning \cite{sun2017complete}, simple neural networks \cite{brutzkus2017globally},  phase retrieval \cite{sun2016geometric}, tensor decomposition \cite{ge2015escaping}, and low rank matrix factorization \cite{bhojanapalli2016global}.

\section{Conclusion}
\label{conclusion}
In this paper, we have constructed an envelope for DYS and established various properties of this envelope. Specifically, there are one-to-one correspondences between the global, local minimizers, critical (stationary) points and strict saddle points of this envelope and those of the original objective. Then, by the Stable-Center Manifold theorem, we have shown that the probability for FBS or DRS to converge from random starting points to strict saddle points of the envelope is zero. If the original objective also satisfies the strict saddle property, we have concluded that, whenever FBS and DRS converge, their iterates will almost always converge to local minimizers.

A limitation of this work lies in its smoothness assumptions. The construction of the envelope requires the Lipschitz differentiability of $g(x)$. Furthermore, twice differentiability of $f(x)$ at specific points is needed for the strict saddle avoidance property of FBS and DRS. It is undoubtedly interesting to investigate the possibility of weakening these assumptions in the future.

\section*{Acknowledgements}
This work is supported in part by NSF Grant DMS-1720237 and ONR Grant N000141712162.

\bibliography{Envelope}
\bibliographystyle{plain}

\end{document}